\documentclass[twoside,11pt]{article}

\usepackage{amsmath,amsthm,amssymb,amsfonts,amsbsy}
\usepackage{longtable}
\usepackage{xspace}
\newcommand{\norm}[1]{\Vert#1\Vert}
\newcommand{\abs}[1]{\left\vert#1\right\vert}
\newcommand{\sog}[1]{\left\{#1\right\}}
\newcommand{\sko}[1]{\left(#1\right)}
\newcommand{\bra}[1]{\left[#1\right]}
\newcommand{\ip}[1]{\left\langle#1\right\rangle}

\newcommand{\eref}[1]{(\ref{#1})}

\newcommand{\Ds}{\displaystyle}
\newcommand{\be}{\begin{equation}}
\newcommand{\ee}{\end{equation}}
\newcommand{\Bth}{\begin{theorem}}
\newcommand{\Eth}{\end{theorem}}
\newcommand{\Bco}{\begin{cor}}
\newcommand{\Eco}{\end{cor}}
\newcommand{\Ble}{\begin{lem}}
\newcommand{\Ele}{\end{lem}}
\newcommand{\Bex}{\begin{exrm}}
\newcommand{\Eex}{\end{exrm}}
\newcommand{\Bde}{\begin{defin}}
\newcommand{\Ede}{\end{defin}}
\newcommand{\Bre}{\begin{rem}}
\newcommand{\Ere}{\end{rem}}
\newcommand{\Bpr}{\begin{prop}}
\newcommand{\Epr}{\end{prop}}
\newcommand{\bez}{\begin{eqnarray*}}
\newcommand{\eez}{\end{eqnarray*}}
\newcommand{\bc}{\begin{center}}
\newcommand{\ec}{\end{center}}
\newcommand{\ra}{\rightarrow}

\newcommand{\half}{\frac{1}{2}}

\newcommand{\pole}{{\bf K}}

\newcommand{\D}{\,{\mathrm d}}
\newcommand{\mr}{{\bf R}}
\newcommand{\mc}{{\bf C}}
\newcommand{\mh}{{\bf H}}

\newcommand{\lmpk}{\ell_{2;\pole}^m\ra\ell_{p;\pole}^n}

\newcommand{\sfer}{{\bf S}(m,\pole)}
\newcommand{\sferk}[1]{{\bf S}(#1,\pole)}
\newcommand{\sferr}[1]{{\bf S}(#1,\mr)}
\newcommand{\sferc}[1]{{\bf S}(#1,\mc)}
\newcommand{\sferh}[1]{{\bf S}(#1,\mh)}
\newcommand{\mpi}{\mathbf{P}}
\newcommand{\km}{\pole^m}
\newcommand{\kpm}{\pole \mpi^{m-1}}
\newcommand{\nq}{\norm{x}}
\newcommand{\ch}{\delta}%field characteristic
\newcommand{\vr}{\varrho}

\newcommand{\nkmp}{N_{\pole}(m,p)}
\newcommand{\uk}{\mathrm{U}(\pole)}

\begin{document}
\theoremstyle{plain}
\newtheorem{theorem}{\sc THEOREM}[section]
\newtheorem{cor}[theorem]{\sc COROLLARY}
\newtheorem{lem}[theorem]{\sc LEMMA}
\newtheorem{defin}[theorem]{\sc DEFINITION}
\newtheorem{prop}[theorem]{\sc PROPOSITION}
\newtheorem{conj}[theorem]{\sc CONJECTURE}
\newtheorem{remark}[theorem]{\sc REMARK}
\newtheorem{ex}[theorem]{\sc EXAMPLE}

\theoremstyle{remark}
\newtheorem{exrm}[theorem]{\sc EXAMPLE}
\newtheorem{rem}[theorem]{\sc REMARK}

\noindent
\parskip=10pt plus 4pt minus 2pt
\baselineskip=15pt ~\\~\\~\\
\begin{flushleft}
{\Large A recursive  construction of projective cubature formulas and related
isometric embeddings~\\~\\  }
{\large YURI I. LYUBICH\footnote{Corresponding author} and OKSANA A. SHATALOVA}\\

{\small\em
Department of Mathematics, Technion, Haifa 32000, Israel\\
e-mail: lyubich@tx.technion.ac.il}\\

{\small\em
Department of Mathematics, Texas A\&M University, College Station, TX 77843, USA\\
e-mail: shatalov@math.tamu.edu} \\
\end{flushleft}
{\footnotesize{\bf Abstract.} A recursive construction is presented for the projective
cubature formulas of index $p$ on the unit spheres
$\sfer\subset\pole^m$ where $\pole$ is $\mr$ or $\mc$, or $\mh$.
This yields a lot of new upper bounds for the minimal number of nodes $n=\nkmp$
in such formulas
or, equivalently, for the minimal $n$ such that there exists an isometric embedding $\lmpk$.

\begin{flushleft}
\footnotesize{\bf 2000 Mathematics Subject Classification:} 46B04, 65D32.\\
{\bf Key words:} cubature formula, Gauss-Jacobi quadrature, isometric embedding \\
\end{flushleft}

\pagestyle{myheadings}
 \markboth{\sc YURI I.LYUBICH AND OKSANA A.SHATALOVA}{}
 \markright{\sc Recursion Construction Of Projective Cubature Formulas}

\section{Introduction and overview}
\label{sec:intro}
\newcommand{\eqsepsp}{\;}
\newcommand{\eqsp}{\quad}
\newcommand{\homopk}{\ensuremath{{\Phi}_{\pole}(m,p)}\xspace}
\newcommand{\homopr}{\ensuremath{{\Phi}_{\mr}(m,p)}\xspace}
\newcommand{\homopc}{\ensuremath{\Phi_{\mc}(m,p)}\xspace}
\newcommand{\homoph}{\ensuremath{\Phi_{\mh}(m,p)}\xspace}

\newcommand{\homok}{\ensuremath{\Phi_{\pole}(m-1,p)}\xspace}
\newcommand{\homor}{\ensuremath{\Phi_{\mr}(m-1,p)}\xspace}
\newcommand{\homorch}[2]{\ensuremath{\Phi_{\mr}(#1,#2)}\xspace}
\newcommand{\homoc}{\ensuremath{\Phi_{\mc}(m-1,p)}\xspace}
\newcommand{\homoh}{\ensuremath{\Phi_{\mh}(m-1,p)}\xspace}
\newcommand{\nmd}{N_{\pole}(m,p)}

 \newcommand{\ps}{\pole\mathrm{P}^{m-1}}

Let $\pole$ be one of three classical fields: $\mr$ (real), $\mc$ (complex), $\mh$ (quaternion).
Its real dimension is
\be\ch=\ch(\pole)=\left\{\begin{array}{ll}1
&(\pole=\mr)\\2&(\pole=\mc)\\4&(\pole=\mh).
\end{array}\right.
\ee
We consider the right $\pole$-linear space $\pole^m$ consisting of the columns
$x=\bra{\xi_i}_1^m$, $\xi_i\in\pole$, $1\leq i\leq m$. This becomes an
{\em Euclidean space} being provided with the inner product
$$\ip{x,y}=\sum_{i=1}^m\bar{\xi}_i\eta_i,\quad x=\bra{\xi_i}_1^m, ~~y=\bra{\eta_i}_1^m,$$
where the bar means the standard conjugation in $\pole$. Obviously,
$$\ip{y,x}=\overline{\ip{x,y}},\qquad\ip{x\alpha,y\beta}=\overline{\alpha}\ip{x,y}\beta.$$
The corresponding {\em Euclidean norm} is the case $p=2$ in the family
$$\norm{x}_p=(\sum_{i=1}^m\abs{\xi_i}^p)^{1/p},\eqsp 1\leq p\leq\infty .$$
With the latter the space  $\pole^m$  is denoted by $l_{p;\pole}^m$, so
the Euclidean space $\pole^m$ is just $l_{2;\pole}^m$. In this case we
will omit the subindex 2 in the notation of the norm.

In $l_{2;\pole}^m$ the unit sphere is
$$ \sfer=\sog{x\in
\pole^m: \nq=1},\quad
\norm{x}=\sqrt{\ip{x,x}}=\Ds{\sqrt{\sum_{i=1}^m\abs{\xi_i}^2}}. $$
Since ${\bf S}(m,\pole)\equiv {\bf S}(\ch m, \mr)$, the topological dimension of
${\bf S}(m,\pole)$ is equal to $\delta m -1$.
In particular, $\sferk{1}=\uk\equiv\sog{\alpha\in\pole,~\abs{\alpha}=1}.$
This is a  multiplicative group acting as $x\mapsto x\alpha$ on $\sfer$. The corresponding
quotient space is the projective space $\kpm$. Its topological dimension is equal to $\ch(m-1)$.
The space $\pole\mathbf{P}^0$ is a singleton.
\Bde \emph{\cite{LSsp}}
Let $p$ be an integer even, $p\geq 2.$
A function $\phi:\pole^m\ra\mc$ belongs to the \emph{class} $\homopk$ if
  \renewcommand{\labelenumi}{\alph{enumi})}
  \begin{description}
  \item[a)] $\phi$ is a homogeneous polynomial of degree $p$ on the real space
$\mr^{\delta m}\equiv(\pole^m)_{\mr}$
      \end{description}
      and
       \begin{description}
  \item[b)] $\phi$ is $\uk$-invariant in the  sense
    that $\phi(x\alpha)=\phi(x)$, $x\in\pole^m,$ $\abs{\alpha}=1,$ or equivalently,
\begin{equation*}
      \phi(x\alpha) = \phi(x)\abs{\alpha}^p,\eqsp x\in\pole^m,\eqsp \alpha\in\pole.
\end{equation*}
  \end{description}
\Ede
As a result, the restriction $\phi|\sfer$ is well defined on $\kpm$. Accordingly,
it is called a {\em polynomial function} on $\kpm$ \cite{LSpf}. For simplicity
we preserve the notation $\phi$ for the projective image of $\phi\in\homopk $. This is
acceptable since the {\em projectivization} is one-to-one.

The simplest example of $\phi\in\homopk$ is $\phi(x)=\norm{x}^p.$
Every $\uk$-invariant (thus even) polynomial $\psi$ of degree $\leq p$ can be
included into $\homopk$ by multiplying each of its
homogeneous component $\psi_d$ by $\norm{\cdot}^{p-d}$, $d=\deg\psi_d=0,2\ldots,p-2,p.$
Since his transformation does not change the restriction $\psi|\sfer$, we have the inclusions
\be
\label{eq:pin}
\Phi_{\pole}(m, d)|\sfer\subset\Phi_{\pole}(m, p)|\sfer \eqsp (d=0,2\ldots,p-2).
\ee

For $\pole =\mr$ the $U(\pole)$-invariance reduces to the {\em central symmetry}, $\phi(-x)=\phi(x)$,
since $U(\mr)={\bf Z_2}$. On the other hand, ${\bf Z_2}\subset U(\pole)$, hence
\be
\label{eq:rin}
\Phi_\pole(m,p)\subset \Phi_{\mr}(\ch m,p).
\ee

Obviously, \homopk is a finite-dimensional complex linear
space. For $\pole=\mr$ this space
consists of all complex-valued homogeneous polynomials of degree $p$ on~$\mr^m$. The
monomials
\begin{equation*}
  \xi_1^{i_1}\ldots\xi_m^{i_m},\quad
  (\xi_k)_1^m\subset\mr^m,
\end{equation*}
with $i_1+\ldots i_m=p$ form a basis of $\homopr$. Accordingly,
\begin{equation}
\label{eq:R}
\dim \homopr= {{m+p-1}\choose{m-1}}.
\end{equation}

In the space \homopc a natural basis consists of all monomials
\begin{equation*}
   \xi_1^{i_1}\cdots\xi_m^{i_m}\bar\xi_1^{j_1}\cdots\bar\xi_m^{j_m},
  \quad \sko{\xi_k}_1^m\subset\mc^m,
\end{equation*}
where $(i_1,\ldots,i_m)$ and $(j_1,\ldots,j_m)$
independently run over all nonnegative $m$-tuples such that
$i_1+\cdots+i_m = j_1+\cdots+j_m=p/2$. Thus, the space $\homopc$
coincides with that of \cite{K}. We have
\begin{equation}
\label{eq:C}\dim \homopc= {{m+p/2-1}\choose{m-1}}^2.
\end{equation}

The structure of \homoph is much more complicated because
of the non-commutativity of the field \mh. The point is that the quaternion monomials
are not $\mathrm{U}(\mh)$-invariant, in general.
However, there exists an alternative way to calculate $\dim\homopk$ for all fields
$\pole$ at once, see \cite{LSsp}. In particular,
\begin{equation}
\label{eq:H}
\dim \homoph= \frac{1}{2m-1}{{2m+p/2-2}\choose{2m-2}}{{2m+p/2-1}\choose{2m-2}}.
\end{equation}

\Bde [ \cite{LSsp}, \cite{LSpf}]
A \emph{projective cubature formula of index} $p$ in $\kpm$ is an identity
\be\label{idodin}
\int_{\sfer}\!\!\!\phi\,\mathrm{d}\sigma_{\delta m-1}=
\sum_{k=1}^n\phi(x_k)\rho_k,~\phi\in\homopk,
\ee
where $\sigma_{\delta m-1}$ is the normalized measure on $\sfer$ induced by the volume in
$\mr^{\delta m}$, the \emph{nodes} $x_k\in\kpm$, all \emph{weights}  $\rho_k>0$ and
their sum is equal to 1.
\Ede

In an equivalent setting all $x_k\in\sfer$ and $x_i\neq x_k\alpha$ for
$\alpha\in U(\pole)$ and $i\neq k$. In this sense $x_k$ are pairwise
\emph{projectively distinct.}

For $\pole=\mr$ the identity \eref{idodin} is a {\em spherical cubature formula of
index} $p$ \cite{GS}, \cite{LV}. In the case of equal weights the set of nodes
of a spherical cubature formula is a \emph{spherical design} \cite{DGS} of the same index.
Similarly, a {\em projective design} over any field $\pole$ can be defined
as the set of nodes of a projective cubature formula with equal weights, c.f. \cite{hoguni}.
Note that {\em a spherical cubature formula is projective if and only if it is
\emph{podal} }\cite{LV}, i.e. there are no pairs of antipodal nodes.

For our purposes it is important that every projective cubature formula of index $p$ is also of
all indices $d=0,2,\cdots p-2$. This immediately follows from \eref{eq:pin} \cite{LSsp,LSpf}.
Hence, a natural symmetrization of a podal spherical cubature formula of index $p$
is an antipodal formula of {\em degree} $p+1$ that means its validity for
{\em all} polynomials on $\mr^m$ of degrees $\leq p+1$.

Now note that the space $\homopk$  contains  all {\em elementary
polynomials} $\phi_{y;p}(x)=\abs{\ip{x,y}}^p$, $y\in\pole^m$. Moreover,
{\em any function $\phi\in\homopk$ is a linear
combination of elementary polynomials} \cite{LSpf}.
For this reason the projective cubature formula \eref{idodin} is equivalent to
the identity
\be\label{iddva}\int_{\sfer}\!\!\!\abs{\ip{x,y}}^p\mathrm{d}\sigma_{\delta m-1}(x)=
\sum_{k=1}^n\abs{\ip{x_k,y}}^p\rho_k,\quad y\in\km.
\ee
On the other hand,
\begin{equation}\label{unocinpr}
\int_{\sfer}\!\!\!\abs{\ip{x,y}}^p\mathrm{d}\sigma_{\delta m-1} (x)\!=
\!\gamma_{_{m,p;\pole}}\norm{y}^p,~~
\gamma_{_{m,p;\pole}}\!=\!{\rm const},\quad y\in\pole^m,
 \end{equation}
see \cite{LSsp}. For $\pole=\mr$ this is the identity applied by
Hilbert \cite{Hilbert} to solve the Waring problem in the number theory.
Irrespective to $\pole$, we call (\ref{unocinpr}) the \emph{Hilbert identity.}

 Comparing \eref{unocinpr} to \eref{iddva} we obtain
 \be\label{idtri}
\sum_{k=1}^n\abs{\ip{u_k,y}}^p=\norm{y}^p,~y\in\km,\ee where
$u_k=x_k\alpha_k$ with some $\alpha_k>0.$ This just means that
\emph{the linear mapping $y\mapsto(\ip{u_k,y})_{k=1}^n$ is an isometric embedding
$\ell_{2;\pole}^m\ra\ell_{p;\pole}^n$}. Moreover, this one is \emph{irreducible}
in the sense that every pair of the vectors $u_i,u_k$ is linearly
independent, in particular, all $u_k\neq 0$.  With any $u_k$'s the identity (\ref{idtri})
can be reduced to a similar identity with some $\widetilde{u_k}$'s,
$1\leq k\leq \widetilde{n}\leq n$, such that the corresponding isometric embedding is irreducible.

Conversely, \emph{every irreducible isometric embedding $\ell_{2;\pole}^m\ra\ell_{p;\pole}^n$
is generated by a projective cubature formula} since
\eref{idtri}\&\eref{unocinpr}$\Rightarrow$\eref{iddva} with $x_k=u_k/\norm{u_k}$ and
$\rho_k=\gamma_{_{m,p;\pole}}\norm{u_k}^p$ .
Thus, we have a 1-1 correspondence between projective cubature formulas of index $p$
with $n$ nodes on $\sfer$ and irreducible isometric embeddings
$\ell_{2;\pole}^m\ra\ell_{p;\pole}^n$.

Note that the image of any isometric embedding $\ell_{2;\pole}^m\ra\ell_{p;\pole}^n$ is an
Euclidean subspace of $\ell_{p;\pole}^n$, and all Euclidean subspaces are of this origin.

For any $(m,p)$ and large $n$ an identity of form (\ref{idtri}) can be derived from the Hilbert
identity directly (i.e. without (\ref{iddva})), see \cite{LSsp} and the references therein.
Accordingly, an isometric embedding $\ell_{2;\pole}^m\ra\ell_{p;\pole}^n$
exists with such $m,p,n$.
{\em The minimal $n$ such that  an isometric embedding $\ell_{2;\pole}^m\ra\ell_{p;\pole}^n$
exists is denoted by $\nmd$.} Every  {\em minimal} isometric embedding
$\ell_{2;\pole}^m\ra\ell_{p;\pole}^n$ (i.e. such that $n=\nmd$) is irreducible, obviously.
Thus, $\nmd$ {\em is also the minimal number of nodes in the projective cubature
formulas of index $p$ on $\sfer$}.

It is known that
\begin{equation}
\label{unotto}
    \nmd\leq\dim\homopk-1,
\end{equation}
see \cite{Lams} and the references therein.
For any fixed $m$ and $p\ra\infty$ the inequality \eref{unotto} combined with the
formulas \eref{eq:R}, \eref{eq:C} and \eref{eq:H} yields the asymptotical upper bound
\begin{equation}
\label{asunotto}
\nmd\lesssim \frac{p^{\ch(m-1)}}{c_m({\pole})}
\end{equation}
where
\begin{equation}
\label{casunotto}
c_m(\mr) =(m-1)!,\eqsp c_m(\mc) =4^{m-1}(m-1)!^2,\eqsp c_m(\mh) =16^{m-1}(2m-1)!(2m-2)!.
\end{equation}

The exact values $\nmd$
are unknown, except for some special cases, see \cite{hoguni}, \cite{K},
\cite{L}, \cite{LSe}, \cite{LV}, \cite{R}. The trivial examples are
\begin{equation}
\label{eq:simp}
N_{\pole}(1,p)=1, \eqsp N_{\pole}(m,2)=m.
\end{equation}
The simplest nontrivial example is $N_{\mr}(2,4)=3$, see \cite{L}. More generally,
\begin{equation}
\label{eq:scocit}
N_{\mr}(2,p)=p/2+1,
\end{equation}
see \cite{LV}, \cite{R}.

From \eref{eq:pin} it follows that
\be
\label{eq:redu}
N_{\pole}(m,p-2)\leq N_{\pole}(m,p).
\ee
Another useful inequality is
\be\label{eq:koly}
N_{\pole}(m,p)\leq N_{\mr}(\ch m,p)\leq N_{\mr}(\ch,p)N_{\pole}(m,p).
\ee
Here the left-hand side follows from \eref{eq:rin} immediately. With $\pole =\mc$
the right-hand side of \eref{eq:koly} follows from \cite{K}, Corollary 3. The proof
of the latter can be adapted to $\pole = \mh$.

In the present paper we construct a recursion with respect to $m$ for the projective
cubature formulas of index $p$ in  $\kpm$. For a large set of pairs $m,p$ this yields
the upper bounds for $\nkmp$ which are {\em effective} in the sense that they are
better than \eref{unotto}.
Later on we call the right-hand side of \eref{unotto} the {\em General Upper Bound},
briefly {\em GUB}. This is a polynomial in $p$ of degree $\ch(m-1)$. It is an open problem to
improve \eref{unotto} in general.

Our Main Theorem is
\Bth
\label{th:scsec}
Let $m\geq 2$, $p\geq 4$. Any projective cubature formula of index $p$ with $n$ nodes
on ${\bf S}(m-1,\pole)$ determines a
projective cubature formula of the same index with $n'$
nodes on ${\bf S}(m,\pole)$ where
\begin{equation}
\label{eq:mainth}
  n'=
\left\{\begin{array}{ll}
    \nu_{\pole}(p)(p/2+1)n, & p\equiv 2\pmod{4} \\
    \nu_{\pole}(p)((p/2)n+1), & p\equiv 0\pmod{4} \
  \end{array}\right.
\end{equation}
and
\begin{equation}
\label{eq:nukp}
  \nu_{\pole}(p)=N_{\mr}(\ch,2\bra{p/4}) =
\left\{\begin{array}{ll}
   N_{\mr}\sko{\ch,p/2-1}, & p\equiv 2\pmod{4} \\
   N_{\mr}\sko{\ch,p/2}, & p\equiv 0\pmod{4}. \
  \end{array}\right.
\end{equation}
\Eth
In fact, $\nu_{\mr}(p) = 1$ and $ \nu_{\mc}(p) =\bra{p/4} +1$ according to \eref{eq:simp} and
\eref{eq:scocit}, respectively. In contrast, for $\nu_{\mh}(p)$ we only have an upper bound
(see \eref{eq:nuovo}), except for $\nu_{\mh}(4)=N_{\mr}(4,2) =4$, see \eref{eq:simp} , and $\nu_{\mh}(8) =N_{\mr}(4,4)= 11$,
see \cite{R}, Proposition 9.26.

In terms of isometric embeddings the Theorem \ref{th:scsec} is reformulated as follows.
\Bth
\label{th:mtise}
Let $m\geq 2$, $p\geq 4$. Any irreducible isometric embedding
$\ell_{2;\pole}^{m-1}\ra\ell_{p;\pole}^n$ determines an  irreducible
isometric embedding $\ell_{2;\pole}^m\ra\ell_{p;\pole}^{n'}$ where
$n'$ is that of \eref{eq:mainth}.
\Eth
Taking $n=N_\pole(m-1,p)$ in \eref{eq:mainth} we obtain
\Bco
\label{cor:forcor41}
The inequality
\begin{equation}
\label{eq:scsedipart}
  \nmd\leq\left\{\begin{array}{ll}
    N_{\mr}\sko{\ch,p/2-1}(p/2+1)N_\pole(m-1,p), & p\equiv 2\pmod{4} \\
    N_{\mr}\sko{\ch,p/2}((p/2)N_\pole(m-1,p)+1), & p\equiv 0\pmod{4}
  \end{array}\right.
\end{equation}
holds.
\Eco
The inequality \eref{eq:scsedipart} being combined with the left-hand side
of \eref{eq:koly} yields
\Bco
\label{cor:prib}
The inequality
      \begin{equation}
\label{eq:scsedi}
  N_{\pole}(m,p)\leq \left\{\begin{array}{ll}
    N_{\mr}\sko{\ch,p/2-1}(p/2+1)N_{\mr}(\ch(m-1),p), & p\equiv 2\pmod{4} \\&\\
   N_{\mr}\sko{\ch,p/2}((p/2)N_{\mr}(\ch(m-1),p)+1), & p\equiv 0\pmod{4} \
  \end{array}\right.
\end{equation}
holds.
\Eco

We prove the Main Theorem in Section \ref{sec:main} using a series of lemmas
from Section \ref{sec:lem}.
The recursion \eref{eq:mainth} corresponds to a partial separation of spherical
coordinates and subsequent applying of some relevant cubature (in particular, quadrature)
formulas for the partial integrals. For the spherical cubature formulas and designs
this way is well known \cite{B0}, \cite{B}, \cite{Bon}, \cite{M}, \cite{RB},
\cite{St}, \cite{W}. The lemmas mentioned above allow us to realize the
recursion in the projective context. For the projective designs our
proof can be adapted by using of a quadrature formula of
Chebyshev type of degree $p/2$ instead of Gauss-Jacobi. This yields a
counterpart of Corollary \ref{cor:forcor41} with an upper bound for
the number of nodes instead of $p/2$.

In Section \ref{sec:appl} we reformulate the Main Theorem for each of three fields
separately and, as a result, explicitly. Then in each case we specify the
range of $m$ where the corresponding upper bound $N_{\pole}(m,p)\leq n$ is effective
for all $p$. In addition, the Main Theorem yields a lot of ``sporadic'' numerical upper
bounds arising  from some known ones. In Section \ref{sec:num} these results are
presented in form of tables.

\section{The lemmas}
\label{sec:lem}
\setcounter{equation}{0}

\begin{lem} \label{sigmalem}
Denote by $\widetilde{\sigma}_{r-1}$ the (non-normalized) surface area
on $\sferr{r}$,  $r\geq 2$.
Let $1\leq l\leq r-1$, and let $x=[\xi_i]_1^r\in
\sferr{r}
$, $y=[\xi_i]_1^l$, $z=[\xi_i]_{l+1}^r$,
$\rho=\|z\|$.With $\widehat{y}=y/\norm{y}$ and
$\widehat{z}=z/\norm{z}$ $(y,z\neq 0)$ the formula
 \be\label{SSAA}
\mathrm{d}\widetilde{\sigma}_{r-1}(x)=
(1-\rho^2)^{\frac{l}{2}-1}\rho^{r-l-1}\mathrm{d}\rho
\mathrm{d}\widetilde{\sigma}_{l-1}(\widehat{y})
\mathrm{d}\widetilde{\sigma}_{r-l-1}(\widehat{z})
\end{equation}
holds (under agreement $\mathrm{d}\widetilde{\sigma}_0(\cdot)=1$).
\end{lem}
\begin{proof}
The column $x$ can be written in the form
 \be\label{invziga2}
x=h(\rho,\widehat{y},\widehat{z})
=\begin{bmatrix}\sqrt{1-\rho^2}\widehat{y}\\ \rho
\widehat{z}\end{bmatrix}.
\end{equation}
Denote by $\theta=(\theta_1,\ldots,\theta_{l-1})$ and $\varphi=(\varphi_1,\ldots,
\varphi_{r-l-1})$ where $\theta_k$ and $\varphi_j$ are the spherical coordinates of
$\widehat{y}\in\sferr{l}$ and $\widehat{z}\in\sferr{r-l}$,
respectively. (For $l=1$ there is no $\theta$, for $l=r-1$ there is no $\varphi$.)
From \eref{invziga2} we obtain the
Jacobi matrix
\bez
J=\frac{\mathcal{D}h(\rho,\widehat{y},\widehat{z})}{\mathcal{D}(\rho,\theta,\varphi)}=
\begin{bmatrix}
-\cfrac{\rho\widehat{y}}{\sqrt{1-\rho^2}}
&\sqrt{1-\rho^2}Y&0\\ \widehat{z}&0&\rho Z
\end{bmatrix},
\eez
where $\bra{\widehat{\xi}_i}_1^l=\widehat{y}$, $\bra{\widehat{\xi}_i}_{l+1}^r=\widehat{z}$,
$$
Y=\bra{\frac{\partial\widehat{\xi_i}}{\partial\theta_k}} _{1\leq i\leq l, 1\leq k\leq l-1} , \eqsp
Z=\bra{\frac{\partial\widehat{\xi_i}}{\partial\varphi_j}}_{l+1\leq i\leq r, 1\leq j\leq r-l-1}.
$$
(There is no $Y$ for $l=1$, no $Z$ for $l=r-1$.)

The corresponding  Gram matrix  is
\be\label{eq:gram}\Gamma=J'J=\begin{bmatrix}(1-\rho^2)^{-1}&0&0
\\ 0&(1-\rho^2)Y'Y&0\\  0&0&\rho^2Z'Z\end{bmatrix}. \ee
where dash means conjugation. Indeed, $\norm{\widehat{y}}^2=\norm{\widehat{z}}^2=1$ and
$$\widehat{y}~'Y=
\sum_{i=1}^l\widehat{\xi}_i\frac{\partial\widehat{\xi}_i}{\partial\theta_k}=
\half\frac{\partial}{\partial\theta_k}\left(\sum_{i=1}^l\widehat{\xi}_i^2\right)
=0,\quad 1\leq k\leq l-1,$$ and   $$
\widehat{z}~^\prime Z=\sum_{i=l+1}^r\widehat{\xi}_i\frac{\partial\widehat{\xi}_i}{\partial
\varphi_j}=\half\frac{\partial }{\partial\varphi_j}\sko{\sum_{i=l+1}^r \widehat{\xi}_i^2}=0,
\quad 1\leq j\leq r-l-1.$$
Note that $G\equiv Y'Y$ and $H\equiv Z'Z$ are the Gram matrices
for the Jacobi matrices $Y$ and $Z$ of the mappings
$(\theta_1,\ldots,\theta_{l-1})\mapsto
(\widehat{\xi}_1,\ldots,\widehat{\xi}_l)$ and
$(\varphi_1,\ldots,\varphi_{r-l-1})\mapsto(\widehat{\xi}_{l+1},\ldots,
\widehat{\xi}_r)$, respectively. From \eref{eq:gram} it follows that
$$\det \Gamma=(1-\rho^2)^{l-2}\rho^{2(r-l-1)}\det G\det H. $$
This results in \eref{SSAA} since
 \bez
\mathrm{d}\widetilde{\sigma}_{r-1}(x)&=&\sqrt{\det
\Gamma}~\mathrm{d}\rho\mathrm{d}\theta_1\ldots
\mathrm{d}\theta_{l-1}\mathrm{d}\varphi_1\ldots
\mathrm{d}\varphi_{r-l-1}
\eez
and  $$\mathrm{d}\widetilde{\sigma}_{l-1}(\widehat{y})=\sqrt{\det
G}~\mathrm{d}\theta_1\ldots
\mathrm{d}\theta_{l-1},\quad \mathrm{d}\widetilde{\sigma}_{r-l-1}(\widehat{z})=\sqrt{\det H}
~\mathrm{d}\varphi_1\ldots
\mathrm{d}\varphi_{r-l-1}.$$
\end{proof}

Now let $x\in\sfer$, $m\geq 2$. Then $x=\eta\oplus z$ where
$\eta\in\pole$ and $z\in\pole^{m-1}$, and then
\begin{equation}\label{eq:cinseu}
 x=\sqrt{1-\rho^2}\theta\oplus\rho w,\quad
  \rho\in\bra{-1,1},
  \quad\theta\in\mathbf{S}(1,\pole)\equiv\sferr{\ch},
  \quad w\in\sferk{m-1}\equiv
  \sferr{\ch (m-1)}.
\end{equation}
Accordingly, we set
\begin{equation}\label
{eq:scsed}\phi(\rho,\theta,w)=\phi(\sqrt{1-\rho^2}\theta\oplus\rho w)
\end{equation}
for a  continuous function $\phi(x)$.
Obviously, $  \phi(-\rho,\theta,-w)=\phi(\rho,\theta,w).$
If $\phi(x)$ is central symmetric, i.e. $\phi(-x)=\phi(x)$,
then $\phi(\rho,-\theta,-w)=\phi(\rho,\theta,w)$. As a result,
$  \phi(-\rho,\theta,w)=\phi(\rho,-\theta,w)$.
Therefore, the $\mathbf{Z}_2$-average with respect to
$\rho$, i.e.
\begin{equation}\label{eq:scsec}
  \widetilde{\phi}(\rho,\theta,w)=
  \half\sko{\phi(\rho,\theta,w)+\phi(-\rho,\theta,w)},
\end{equation}
coincides with the $\mathbf{Z}_2$-average with respect to
$\theta$:
\begin{equation}\label{eq:scses}
  \widetilde{\phi}(\rho,\theta,w)=
  \half\sko{\phi(\rho,\theta,w)+\phi(\rho,-\theta,w)}.
\end{equation}

Now we consider the integral
\begin{equation}
\label{eq:scsese}
I_{\phi}(w)= \int_{\mathbf{S}(1,\pole)}\D\sigma_{\ch-1}(\theta)
  \int_0^1\phi(\rho,\theta,w)\pi(\rho)\D\rho
\end{equation}
with any integrable $\pi(\rho)$.
\Ble
\label{lem:scseu}
If $\phi(x)$ is central symmetric then $I_{\phi}(w)=I_{\widetilde{\phi}}(w)$.
\Ele
\begin{proof}
This follows from \eref{eq:scses} since the measure
$\sigma_{\ch-1}(\theta)$ is central symmetric.
\end{proof}
\Ble\label{lem:scsed}
If $\phi(x)$ is $\uk$-invariant then $I_{\phi}(w)$ is also $\uk$-invariant.
\Ele
\begin{proof}
 From \eref{eq:scsed} it follows that
$  \phi(\rho,\theta\alpha,w\alpha)=\phi(\rho,\theta,w)$
for all $\alpha\in\uk$. On the other hand, the measure $\sigma_{\ch-1}(\theta)$
is $\uk$-invariant.
\end{proof}

Actually, only the functions $\phi(x)$ from $\homopk$ are
needed for our purposes.
 \Ble\label{lem:scset}
If $\phi(x)$ belongs to $\homopk$ then the function $I_{\phi}(w)$ belongs to
$\homok|{\bf S}(m-1,\pole)$.
 \Ele
\begin{proof}
In view of the Lemma \ref{lem:scsed} and inclusion \eref{eq:pin} we only
have to prove that $I_{\phi}(w)$ is the restriction to the unit sphere of a polynomial
of degree $\leq p$ on $\mr^{\ch(m-1)}$. Since $\homopk = {\rm Span}\{\phi_{y;p}:y\in\pole^m\}$
and since the mapping $\phi\mapsto I_{\phi}$ is linear,
we can assume that $\phi(x) =  \phi_{y;p}(x)=\abs{\ip{x,y}}^p$, $y\in\pole^{m}$.
Let $y=\xi\oplus v$ where $\xi\in\pole$, $v\in\pole^{m-1}$. Then by \eref{eq:scsed}
\begin{equation}
\label{eq:scsen}
  \begin{split}
  \phi_{y;p}(\rho,\theta,w)=&\abs{\sqrt{1-\rho^2}\overline{\theta}\xi
  +\rho\ip{w,v}}^p \\=&\sko{(1-\rho^2)\abs{\xi}^2+\rho^2\abs{\ip{w,v}}^2 +
  2\rho\sqrt{1-\rho^2}\Re\mathrm{e}\sko{\overline{\xi}\theta\ip{w,v}}}^{p/2}.
  \end{split}
\end{equation}
With fixed $\rho$ and $\theta$ let us consider the right-hand side of \eref{eq:scsen} as a function
of  $w\in\mr^{\ch(m-1)}$. This is a polynomial of degree $\leq p$. Therefore, such is $I_{\phi}(w)$
obtained by substitution of \eref{eq:scsen} into the integral \eref{eq:scsese}.
\end{proof}
The last lemma we need is
 \Ble\label{lem:scseq}
If $\phi(x)$ belongs to $\homopk$ then with a fixed $w$ the function
$\widetilde{\phi}(\rho,\theta,w)$ defined by \eref{eq:scsec}
is a linear combination of functions of form
$f(\rho^2)\sko{\ip{\theta,\zeta}_{\mr}}^{2q}$
where $f$ is a  polynomial of degree $\leq p/2$, $0\leq q\leq\bra{p/4}$,
$\zeta\in\pole$,  ${\ip{\theta,\zeta}_{\mr}} =\Re\mathrm{e}(\overline{\theta}\zeta)$.
\Ele
\begin{proof}
As before, it suffices to consider $\phi =\phi_{y;p}$, so we can use \eref{eq:scsen}.
Note that
$$
\Re\mathrm{e}\sko{\overline{\xi}\theta\ip{w,v}} =\Re\mathrm{e}\sko{\ip{v,w}\overline{\theta}\xi}
=\Re\mathrm{e}\sko{\overline{\theta}\xi\ip{v,w}} = \ip{\theta,\zeta}_{\mr}
$$
where $\zeta=\xi\ip{v,w}$. We have
\begin{equation*}
  \phi(\rho,\theta,w)=
  \sko{A(\rho^2)+B(\rho^2)\mathrm{sign}(\rho)\ip{\theta,\zeta}_{\mr}}^{p/2}
\end{equation*}
where
\begin{equation*}
  A(t)=\abs{\xi}^2(1-t)+\abs{\ip{w,v}}^2 t,\quad B(t)=\sqrt{4t(1-t)}.
\end{equation*}
Hence,
\begin{equation*}
  \phi(\rho,\theta,w)=\sum_{k=0}^{p/2}\binom{p/2}{k}A(\rho^2)^{p/2-k}
  B(\rho^2)^k\sko{\mathrm{sign}(\rho)}^k(\ip{\theta,\zeta}_{\mr})^k,
\end{equation*}
and then \eref{eq:scsec} yields
\begin{equation*}
  \widetilde{\phi}(\rho,\theta,w)=
  \sum_{q=0}^{\bra{p/4}}\binom{p/2}{2q}A(\rho^2)^{p/2-2q}
  B(\rho^2)^{2q}(\ip{\theta,\zeta}_{\mr})^{2q}.
\end{equation*}
It remains to note that $A(t)^{p/2-2q} B(t)^{2q}$ is a polynomial of degree $\leq p/2$
for every $q\leq\bra{p/4}.$
   \end{proof}

\section{Proof of the Main Theorem}
\label{sec:main}
 \setcounter{equation}{0}
Let $\phi\in\homopk$, $x\in\sfer,~\phi(x)=\phi(\rho,\theta,w)$ as in \eref{eq:scsed}.
According to Lemma \ref{sigmalem} with $r=\ch m$ and
$l=\ch$, we have
\begin{equation*}
\int_{\sfer}\phi(x)\D\sigma_{r-1}(x)=
\int_{\sferk{m-1}}\D\sigma_{r-\ch-1}(w)\int_{\mathbf{S}(1,\pole)}
\D\sigma_{\ch-1}(\theta) \int_0^1
\phi(\rho,\theta,w)\pi_{\alpha,\beta}(\rho)\D\rho
\end{equation*}
where
\begin{equation*}
  \pi_{\alpha,\beta}(\rho)=C\rho^{2\alpha+1}(1-\rho^2)^{\beta},
\quad \alpha=\ch(m-1)/2,\quad \beta=\ch/2-1,
\end{equation*}
the constant $C=C_{r,\ch}$ comes from the normalization of the areas in \eref{SSAA}:
\begin{equation*}
\int_0^1\pi_{\alpha,\beta}(\rho)\D\rho =1.
\end{equation*}
By \eref{eq:scsese} and Lemma \ref{lem:scseu} we get
\begin{equation*}
  \int_{\sfer}\phi(x)\D\sigma_{r-1}(x)=\int_{\sferk{m-1}}\D\sigma_{r-\ch-1}(w)
\int_{\mathbf{S}(1,\pole)}\D\sigma_{\ch-1}(\theta) \int_0^1
\widetilde{\phi}(\rho,\theta,w)\pi_{\alpha,\beta}(\rho)\D\rho.
\end{equation*}

Lemma \ref{lem:scset} allows us to apply a
projective cubature formula of index $p$ on
$\sferk{m-1}$ existing by
assumption. If its nodes and weights are $w_i$ and
$\lambda_i$, $1\leq i\leq n$, respectively, then
\begin{equation}\label{eq:scsequ}
   \int_{\sfer}\phi(x)\D\sigma_{r-1}(x)=\sum_{i=1}^n \lambda_i
   \int_{\mathbf{S}(1,\pole)}\D\sigma_{\ch-1}(\theta) \int_0^1
\widetilde{\phi}(\rho,\theta,w_i)\pi_{\alpha,\beta}(\rho)\D\rho.
\end{equation}

By Lemma \ref{lem:scseq} the integrals against $\D\sigma_{\ch-1}(\theta)$ in \eref{eq:scsequ}
can be calculated by a podal spherical cubature formula of index $2\bra{p/4}$ on
$\mathbf{S}(1,\pole)\equiv\sferr{\ch}$.
The minimal number of nodes in such a formula is
\be
\label{eq:minu}
\nu = N_{\mr}(\ch,2\bra{p/4})=
\left\{\begin{array}{ll}
    N_{\mr}\sko{\ch,p/2-1}, & p\equiv 2\pmod{4} \\
    N_{\mr}\sko{\ch,p/2}, & p\equiv 0\pmod{4}.
  \end{array}\right.
\ee
As a result,
\begin{equation}\label{eq:scsequi}
  \int_{\sfer}\phi(x)\D\sigma_{r-1}(x)=\sum_{i=1}^n\sum_{j=1}^{\nu}
  \lambda_i\mu_j \int_0^1
\widetilde{\phi}(\rho,\theta_j,w_i)\pi_{\alpha,\beta}(\rho)\D\rho
\end{equation}
where $\theta_j$ and $\mu_j$ are the corresponding nodes and weights.

Now we consider the integral
\begin{equation*}
%\label{eq:scsesed}
  \int_0^1
  f(\rho^2)\pi_{\alpha,\beta}(\rho)\D\rho=\int_0^1
  f(\tau)\chi_{\alpha,\beta}(\tau)\D\tau
\end{equation*}
where $f$ is a polynomial of degree $\leq p/2$ and
\begin{equation*}
%\label{eq:scseds}
\chi_{\alpha,\beta}(\tau)=\cfrac{\pi_{\alpha,\beta}(\sqrt{\tau})}{2\sqrt{\tau}}=
\half C\tau^{\alpha}(1-\tau)^{\beta},\eqsp\int_0^1\chi_{\alpha,\beta}(\tau)\D\tau =1.
\end{equation*}
Assume that $p\equiv 2\pmod{4}$, i.e. $p/2$ is odd. Since  $\deg f\leq p/2=2(p+2)/4-1$,
the classical Gauss-Jacobi quadrature formula yields
\begin{equation}\label{eq:scsedot}
\int_0^1f(\tau)\chi_{\alpha,\beta}(\tau)\D\tau=\sum_{k=1}^{(p+2)/4}\omega_kf(\tau_k)
\end{equation}
with relevant nodes and weights, see \cite{Sz}, Theorems 3.4.1 and 3.4.2. Therefore,
\begin{equation*}
%\label{eq:scsedn}
  \int_0^1
  f(\rho^2)\pi_{\alpha,\beta}(\rho)\D\rho=\sum_{k=1}^{(p+2)/4}\omega_k
  f(\rho_k^2),\quad \rho_k=\sqrt{\tau_k}.
\end{equation*}
By Lemma \ref{lem:scseq}
\begin{equation}\label{eq:scsew}
  \begin{split}
  \int_0^1
\widetilde{\phi}(\rho,\theta_j,w_i)\pi_{\alpha,\beta}(\rho)\D\rho=&
\sum_{k=1}^{(p+2)/4}\omega_k
\widetilde{\phi}(\rho_k,\theta_j,w_i) \\ =&\half
\sum_{k=1}^{(p+2)/4}\omega_k\Bigl(\phi(\rho_k,\theta_j,w_i)+
\phi(\rho_k,-\theta_j,w_i)\Bigr)
  \end{split}
\end{equation}
for all $1\leq i\leq n$, $1\leq j\leq\nu$. The substitution from
\eref{eq:scsew} into \eref{eq:scsequi} yields
\begin{equation}\label{eqq:scsewud}
  \int_{\sfer}\phi(x)\D\sigma_{r-1}(x)=\sum_{i=1}^n\sum_{j=1}^{\nu}
  \sum_{k=1}^{(p+2)/4}\vr_{ijk}\sko{\phi(x_{ijk}^+)+\phi(x_{ijk}^-)}
\end{equation}
where
\begin{equation}\label{eq:scsewdd}
  x_{ijk}^{\pm}=\pm \theta_j\sqrt{1-\rho_k^2}\oplus\rho_k
  w_i,\quad \vr_{ijk}=\half \lambda_i\mu_j\omega_k.
\end{equation}
The number of nodes  $x_{ijk}^{\pm}$ is
\be
n' = (p/2+1)\nu n = N_{\mr}(\ch, p/2-1)(p/2+1)n
\ee
according to \eref{eq:minu}.

Now let $p\equiv 0\pmod{4}$, i.e. let $p/2$ be even. In this case, instead of
\eref{eq:scsedot}, we use its Markov's modification (see \cite{M}, formula (1.16)):
\begin{equation}\label{eq:scsesu}
  \int_0^1
  f(\tau)\chi_{\alpha,\beta}(\tau)\D\tau=\omega_0f(0)+\sum_{k=1}^{p/4}
  \omega_k f(\tau_k).
\end{equation}
This is valid for all polynomials $f$ of $\deg f\leq 2(p/4)=p/2$.
(Of course, the nodes and the weights in \eref{eq:scsesu} are different from those of
\eref{eq:scsedot}.) As before,
\begin{equation*}
%\label{eq:scsewd}
  \int_0^1
\widetilde{\phi}(\rho,\theta_j,w_i)\pi_{\alpha,\beta}(\rho)\D\rho=
\omega_0\phi(0,\theta_j,w_i)  +\half
\sum_{k=1}^{p/4}\omega_k\Bigl(\phi(\rho_k,\theta_j,w_i)+
\phi(\rho_k,-\theta_j,w_i)\Bigr)
\end{equation*}
and then
\begin{equation}
\label{eq:scsewt}
  \int_{\sfer}\phi(x)\D\sigma_{r-1}(x)=\sum_{j=1}^{\nu}\vr_{j}\phi(x_j)
  +\sum_{i=1}^n\sum_{j=1}^{\nu}
  \sum_{k=1}^{p/4}\vr_{ijk}\sko{\phi(x_{ijk}^+)+\phi(x_{ijk}^-)}
\end{equation}
where
\begin{equation}
\label{eq:nno}
x_j=\theta_j\oplus 0,\eqsp \vr_{j}=\mu_j\omega_0\sum_{i=1}^n\lambda_i =\mu_j\omega_0,
\end{equation}
the rest of nodes and weights is determined as in \eref{eq:scsewdd}.
Now the total number of nodes is
\be
n' = \nu + (p/2)\nu n = N_{\mr}(\ch, p/2)((p/2)n+1)
\ee
according to \eref{eq:minu} again.

It remains to note that in each of formulas \eref{eqq:scsewud} and \eref{eq:scsewt}
the nodes are projectively distinct. \hfill$\square$

\section{Some applications}
\label{sec:appl}
 \setcounter{equation}{0}

Further $m\geq 2$, $p\geq 4$ as in the Main Theorem. It is convenient to set
$p=2s$, so $s$ is an integer, $s\geq 2$.

Let us start with $\pole = \mc$. In this case the Main Theorem takes the form of
\Bth
\label{cor:scsese}
Any projective cubature formula of index $2s$ with $n$
nodes on $\sferc{m-1}$ determines a
projective cubature formula of the same index with $n'$
nodes on $\sferc{m}$  where
\begin{equation}\label{eq:scsews}
  n'=\left\{\begin{array}{ll}
    \dfrac{(s+1)^2}{2} n, & s\equiv 1\pmod{2} \\~\\
    \dfrac{s+2}{2}(sn+1), & s\equiv 0\pmod{2}. \
  \end{array}\right.
\end{equation}
\Eth
\begin{proof}
By \eref{eq:nukp} and \eref{eq:scocit}
\begin{equation*}
\nu_{\mc}(2s)=N_{\mr}(2,2\bra{s/2})= \bra{s/2}+1=
\left\{\begin{array}{ll}
    \dfrac{s+1}{2}, & s\equiv 1\pmod{2} \\&\\
    \dfrac{s+2}{2}, & s\equiv 0\pmod{2}. \
  \end{array}\right.
\end{equation*}
\end{proof}
The Corollary \ref{cor:forcor41} reduces to
\Bco
\label{cor:mtc}
The inequality
\begin{equation}
\label{eq:scowc}
  N_{\mc}(m,2s)\leq\left\{\begin{array}{ll}
 \dfrac{(s+1)^2}{2} N_{\mc}(m-1,2s), & s\equiv 1\pmod{2} \\&\\
    \dfrac{s+2}{2}\left(sN_{\mc}(m-1,2s)+1\right), & s\equiv 0\pmod{2}, \
  \end{array}\right.
\end{equation}
holds.
\Eco
In particular,
\begin{equation}
\label{eq:ubc2}
  N_{\mc}(2,2s)\leq
\left\{\begin{array}{ll}
\dfrac{(s+1)^{2}}{2},& s\equiv 1\pmod{2},\\&\\
\dfrac{(s+2)(s+1)}{2} , &s\equiv 0\pmod{2},
\end{array}\right.
\end{equation}
since $N_{\mc}(1,2s) =1$. Asymptotically,
\begin{equation}
\label{eq:scown}
  N_{\mc}(2,2s)\lesssim\frac{1}{2}s^2,\quad s\ra\infty.
\end{equation}

Taking $m=3$ in \eref{eq:scowc} and using \eref{eq:ubc2} we obtain
\begin{equation}
\label{eq:ubc3}
  N_{\mc}(3,2s)\leq
\left\{\begin{array}{ll}
\dfrac{(s+1)^{4}}{4},& s\equiv 1\pmod{2},\\&\\
\\\dfrac{s+2}{2}\left(\dfrac{(s+2)(s+1)s}{2}+1\right) , &s\equiv 0\pmod{2},
\end{array}\right.
\end{equation}
whence
\begin{equation}
\label{eq:scownn}
  N_{\mc}(3,2s)\lesssim\frac{1}{4}s^4,\quad s\ra\infty.
\end{equation}

{\em The upper bounds \eref{eq:ubc2} and \eref{eq:ubc3} are effective.}
Indeed, for $\pole = \mc$ the cases $m=2,3$ in GUB (i.e., in \eref{unotto}) are
\begin{equation}
\label{eq:gubc}
N_{\mc}(2,2s)\leq (s+1)^2, \eqsp  N_{\mc}(3,2s)\leq \dfrac{(s+2)^2(s+1)^2}{4},
\end{equation}
that is worse than \eref{eq:ubc2} and \eref{eq:ubc3}, respectively.
Asymptotically, \eref{eq:scown} also remains effective, i.e. better
than what the first inequality \eref{eq:gubc} implies. However, \eref{eq:scownn} coincides with
the corresponding consequence of \eref{eq:gubc}. (Clearly, it cannot be worse.)

The next iteration of \eref{eq:scowc} yields an ineffective upper bound for
$N_{\mc}(m,2s)$, $m\geq 4$. However, for some $s$ the effectiveness may be reached  by
using a more precise bound (or an exact value, if any) for $N_{\mc}(m-1,2s)$
in \eref{eq:scowc}. Also, some effective bounds can be improved in this way.
In Section 5 the reader can find a lot of examples of this approach (for all three fields).
One of them is below.
\begin{ex}
{\em From the known (see \cite{LSe} ) equality $N_{\mc}(2,8)=10$
it follows that
\begin{equation*}
N_{\mc}(3,8)\leq 3(4N_{\mc}(2,8)+1) = 123,
\end{equation*}
while \eref{eq:ubc3} yields $N_{\mc}(3,8)\leq 183$.}
\end{ex}

The following is the iterated form of Theorem \ref{cor:scsese}.
\Bth
\label{thm:itc}
Any projective cubature formula of index $2s$ with $n$
nodes on $\sferc{m-1}$ determines a
projective cubature formula of the same index with $n^{(l)}$
nodes on $\sferc{m + l -1}$, where $l\geq 0$ and
\begin{equation}
\label{eq:utc}
n^{(l)} = \left\{\begin{array}{ll}
\dfrac{(s+1)^{2l}}{2^{l}}n, & s\equiv 1\pmod{2},\\&\\
\left(A\dfrac{(s+2)^ls^l}{2^l} +B\right)n, &s\equiv 0\pmod{2},
\end{array}\right.
\end{equation}
with
\be
A=\frac{(s+3)s}{(s+2)s-2}\eqsp,\eqsp B=- \frac{s+2}{(s+2)s-2}.
\ee
\Eth
\begin{proof}
For any $s$ the sequence on the right-hand side
of \eref{eq:utc} satisfies the recurrent relation  \eref{eq:scsews},
and $N^0 = N$ since $A+B=1$.
\end{proof}
\Bco
\label{cor:mtcl}
The inequality
\begin{equation}
\label{eq:scowcc}
  N_{\mc}(m+l-1,2s)\leq\left\{\begin{array}{ll}
 \dfrac{(s+1)^{2l}}{2^l} N_{\mc}(m-1,2s), & s\equiv 1\pmod{2} \\&\\
\left(A\dfrac{(s+2)^ls^l}{2^l} +B\right) N_{\mc}(m-1,2s)
& s\equiv 0\pmod{2}. \
  \end{array}\right.
\end{equation}
holds.
\Eco

Now let us proceed to $\pole =\mr$. In this case we have
\Bth
\label{cor:scses}
Any podal spherical cubature formula of index $2s$ with $n$
nodes on $\sferr{m-1}$ determines a
podal spherical cubature formula of the same index with
$n'$ nodes on $\sferr{m}$ where
\begin{equation}
\label{eq:scsewq}
  n'=\left\{\begin{array}{ll}
    (s+1)n, & s\equiv 1\pmod{2} \\&\\
    sn+1, & s\equiv 0\pmod{2}.
  \end{array}\right.
\end{equation}
\Eth
\begin{proof}
$\nu_{\mr}(2s) = N_{\mr}(1, 2\bra{s/2}) = 1$.
\end{proof}
\Bco
\label{cor:mtr}
\begin{equation}
\label{eq:scociu}
  N_{\mr}(m,2s)\leq\left\{\begin{array}{ll}
    (s+1)N_{\mr}(m-1,2s), & s\equiv 1\pmod{2} \\&\\
    sN_{\mr}(m-1,2s)+1, & s\equiv 0\pmod{2}. \
  \end{array}\right.
\end{equation}
\Eco

For $m=2$ both inequalities \eref{eq:scociu} reduce to $N_{\mr}(2,2s)\leq s+1$.
(In fact, $N_{\mr}(2,2s)=s+1$, see \eref{eq:scocit}.) Hence,
\begin{equation}
\label{eq:scociq}
N_{\mr}(3,2s)\leq\left\{\begin{array}{ll}
    (s+1)^2, & s\equiv 1\pmod{2} \\&\\
    s^2+s+1& s\equiv 0\pmod{2},
  \end{array}\right.
\end{equation}
thus
\begin{equation}\label{eq:scocic}
  N_{\mr}(3,2s)\lesssim s^2,\quad s\ra\infty.
\end{equation}
The next iteration yields
\begin{equation}
\label{eq:r42sp}
N_{\mr}(4,2s)\leq\left\{\begin{array}{ll}
    (s+1)^3, & s\equiv 1\pmod{2} \\&\\
   (s^2+1)(s+1)& s\equiv 0\pmod{2}.
  \end{array}\right.
\end{equation}
However, the latter can be improved by means of the inequality
\begin{equation}
\label{eq:nuov}
    N_{\mr}(2m,2s)\leq (s+1)N_{\mc}(m,2s)
\end{equation}
which is just the case $\ch = 2$ on the right-hand side of \eref{eq:koly}. Indeed,
\begin{equation}
\label{eq:nuovo}
N_{\mr}(4,2s)\leq (s+1)N_{\mc}(2,2s)\leq
\left\{\begin{array}{ll}
   \dfrac{(s+1)^3}{2} , & s\equiv 1\pmod{2} \\&\\
    \frac{(s+2)(s+1)^2}{2}, & s\equiv 0\pmod{2}. \
  \end{array}\right.
\end{equation}
This is better than \eref{eq:r42sp},
except for the case $s=2$, i.e. for $N_{\mr}(4,4)=11$. From \eref{eq:nuovo} we get
\begin{equation}
\label{eq:asr4}
  N_{\mr}(4,2s)\lesssim \dfrac {s^3}{2},\quad s\ra\infty,
\end{equation}
instead of $N_{\mr}(4,2s)\lesssim s^3$ that follows from \eref{eq:r42sp}.

Similarly,
\be
\label{eq:nuovoo}
N_{\mr}(6,2s)\leq (s+1)N_{\mc}(3,2s)\leq
\left\{\begin{array}{ll}
\dfrac{(s+1)^{5}}{4},& s\equiv 1\pmod{2},\\&\\
\\\dfrac{(s+2)(s+1)}{2}\left(\dfrac{(s+2)(s+1)s}{2}+1\right) , &s\equiv 0\pmod{2},
\end{array}\right.
\ee
by \eref{eq:ubc3}. Hence,
\begin{equation}
\label{eq:asr6}
  N_{\mr}(6,2s)\lesssim \dfrac {s^5}{4},\quad s\ra\infty.
\end{equation}

In addition, from \eref{eq:scociu} and \eref{eq:nuovo} it follows that
\be
\label{eq:uovo}
N_{\mr}(5,2s)\leq
\left\{\begin{array}{ll}
   \dfrac{(s+1)^4}{2} , & s\equiv 1\pmod{2} \\&\\
    \dfrac{(s+2)(s+1)^2s}{2} + 1, & s\equiv 0\pmod{2}, \
  \end{array}\right.
\ee
hence,
\begin{equation}
\label{eq:asr5}
  N_{\mr}(5,2s)\lesssim \dfrac {s^4}{2},\quad s\ra\infty.
\end{equation}

{\em All upper bounds for $ N_{\mr}(m,2s)$, $3\leq m\leq 6$, obtained above are effective,
even asymptotically, c.f. \eref{asunotto}}.

The $\mr$-counterpart of Theorem \ref{thm:itc} looks simpler.
\Bth
\label{cor:itr}
Any podal spherical cubature formula of index $2s$ with $n$
nodes on $\sferr{m-1}$ determines a
podal spherical cubature formula of the same index with
$n^{(l)}$ nodes on $\sferr{m+l-1}$ where $l\geq 0$ and
\begin{equation}
\label{eq:scsewql}
  n^{(l)}=\left\{\begin{array}{ll}
    (s+1)^l n, & s\equiv 1\pmod{2} \\&\\
    s^ln+\dfrac{s^l-1}{s-1} , & s\equiv 0\pmod{2}.
  \end{array}\right.
\end{equation}
\Eth
\begin{proof}
Induction on $l$.
\end{proof}
\Bco
\label{cor:mtrn}
\begin{equation}\label{eq:scociuu}
  N_{\mr}(m+l-1,2s)\leq\left\{\begin{array}{ll}
    (s+1)^lN_{\mr}(m-1,2s), & s\equiv 1\pmod{2} \\&\\
    s^lN_{\mr}(m-1,2s)+\dfrac{s^l-1}{s-1} , & s\equiv 0\pmod{2}. \
  \end{array}\right.
\end{equation}
\Eco

It remains to consider the case $\pole=\mh$.
\Bth
\label{th:scsese}
Any projective cubature formula of index $2s$ with $n$ nodes
on ${\bf S}(m-1,\mh)$ determines a
projective cubature formula of the same index with $n'$
nodes on ${\bf S}(m,\mh)$ where
\begin{equation}
\label{eq:scsedie}
  n'=
\left\{\begin{array}{ll}
   N_{\mr}\sko{4,s-1}(s+1)n, & s\equiv 1\pmod{2} \\
   N_{\mr}\sko{4,s}(sn+1), & s\equiv 0\pmod{2}. \
  \end{array}\right.
\end{equation}
\Eth
\begin{proof}
We have
\be
\label{eq:info}
\nu_{\mh}(2s)=N_{\mr}(4,2\bra{s/2})=
\left\{\begin{array}{ll}
   N_{\mr}\sko{4,s-1}, & s\equiv 1\pmod{2} \\
   N_{\mr}\sko{4,s}, & s\equiv 0\pmod{2}. \
  \end{array}\right.
\ee
\end{proof}

\Bco
\label{cor:mth}
The inequality
\begin{equation}
\label{eq:scsedip}
  N_{\mh}(m,2s)\leq
\left\{\begin{array}{ll}
    N_{\mr}\sko{4,s-1}(s+1)N_{\mh}(m-1,2s), & s\equiv 1\pmod{2} \\
    N_{\mr}\sko{4,s}(sN_{\mh}(m-1,2s)+1), & s\equiv 0\pmod{2}.
  \end{array}\right.
\end{equation}
holds.
\Eco
The exact values of $N_{\mr}\sko{4,2\bra{s/2}}$ are unknown, except
for the cases $s=2$ and $s=4$ when $N_{\mr}\sko{4,2} = 4$ and
$N_{\mr}\sko{4,4} = 11$, respectively. However, we can use  the upper bound \eref{eq:nuovo}.
\Bth
\label{th:scsese1}
Any projective cubature formula of index $2s$ with $n$ nodes on $\sferh{m-1}$ determines a
projective cubature formula of the same index with $n'$ nodes on $\sferh{m}$ where
\begin{equation}
\label{eq:hrec}
16n'\leq
\left\{\begin{array}{ll}
(s+1)^4 n, & s\equiv 3\pmod{4} \\&\\
(s+3)(s+1)^3n, & s\equiv 1\pmod{4} \\&\\
(s+2)^3(sn+1), & s\equiv 2\pmod{4} \\&\\
(s+4)(s+2)^2(sn+1), & s\equiv 0\pmod{4}. \
\end{array}\right.
\end{equation}
\Eth
\begin{proof}
If $s\equiv 0\pmod{4}$ then $s/2\equiv 0\pmod{2}$ and \eref{eq:nuovo} yields
\be
\label{eq:nr4s0}
N_{\mr}(4,s)\leq \frac{(s/2+2)(s/2+1)^2}{2}=  \frac{(s+4)(s+2)^2}{16}.
\ee
Now let $s\equiv 1\pmod{4}$. Then $s-1\equiv 0\pmod{2}$ and \eref{eq:nr4s0}
turns into
\be
\label{eq:nr4s1}
N_{\mr}(4,s-1)\leq \frac{(s+3)(s+1)^2}{16}.
\ee
Similarly, if $s\equiv 2\pmod{4}$ then $s/2\equiv 1\pmod{2}$, hence
\be
\label{eq:nr4s2}
N_{\mr}(4,s)\leq\frac{(s/2+1)^3}{2} =\frac{(s+2)^3}{16}
\ee
by \eref{eq:nuovo}. Finally,  if $s\equiv 3\pmod{4}$ then $s-1\equiv 2\pmod{2}$, hence
\be
\label{eq:nr4s3}
N_{\mr}(4,s-1)\leq \frac{(s+1)^3}{16}.
\ee
by \eref{eq:nr4s2}. It remains to substitute the inequalities \eref{eq:nr4s0}-\eref{eq:nr4s3}
into \eref{eq:scsedie}.
\end{proof}
\Bco
\label{cor:nhin}
The inequality
\begin{equation}
\label{eq:nhmp}
16N_{\mh}(m,2s)\leq
\left\{\begin{array}{ll}
(s+1)^4 N_{\mh}(m-1,2s), & s\equiv 3\pmod{4} \\&\\
(s+3)(s+1)^3N_{\mh}(m-1,2s), & s\equiv 1\pmod{4} \\&\\
(s+2)^3(sN_{\mh}(m-1,2s)+1), & s\equiv 2\pmod{4} \\&\\
(s+4)(s+2)^2(sN_{\mh}(m-1,2s)+1), & s\equiv 0\pmod{4} \
\end{array}\right.
\end{equation}
holds.
\Eco
In particular,
\begin{equation}
\label{eq:nh2p}
16N_{\mh}(2,2s)\leq
\left\{\begin{array}{ll}
(s+1)^4, & s\equiv 3\pmod{4} \\&\\
(s+3)(s+1)^3, & s\equiv 1\pmod{4} \\&\\
(s+2)^3(s+1), & s\equiv 2\pmod{4} \\&\\
(s+4)(s+2)^2(s+1), & s\equiv 0\pmod{4} \
\end{array}\right.
\end{equation}
 Asymptotically,
\begin{equation}
\label{eq:ash}
  N_{\mh}(2,2s)\lesssim\frac{1}{16}s^4,\quad s\ra\infty.
\end{equation}

{\em The upper bounds \eref{eq:nh2p} are effective, even asymptotically}.

\section{The numerical results}
\label{sec:num}
 \setcounter{equation}{0}

In this section we present the tables of effective numerical upper bounds
for $N_{\pole}(m,p)$ obtained by the recursion combined with other tools, if any.
We do not include those of bounds which are
worse than known once.
Of course, it would be meaningless to tabulate the general inequalities like \eref{eq:ubc2}.
However, some their numerical consequences are presented for the reader convenience.

The tables are organized as follows.
The Table 1 contains those known equalities of form $n=N_{\pole}(m,p)$ which are
used as the starting data (the {\em input}) for the recursion. The equalities are enumerated as e1, e2,...
Similarly, in the Table 2 the input inequalities $N_{\pole}(m,p)\leq n$
are enumerated as i1, i2,... The Tables 3, 4, 5 contain the resulting upper bounds for
$\pole=\mr$, $\mc$, $\mh$, respectively, enumerated as r0, r1,...
within each table. In every of these tables the enumeration is established
in ascending order of $m$.
The effectiveness of all results is demonstrated by including of the
corresponding GUB \eref{unotto} into the tables.
Several cases of known upper bounds which are weaker than ours
are mentioned after the tables.

All input data are provided with the bibliographic references. For all results
we refer to the input data and to the general facts from Section \ref{sec:appl}
and, sometimes, from Section \ref{sec:intro}. Also, there are some cross-references
between the Tables of results.

Let us remember three equivalent interpretations of the inequality
$N_{\pole}(m,p)\leq n$.

a) {\em There exists a projective cubature formula of index $p$ with
  $n$ nodes on the sphere $\sferk{m}$.}

b) {\em There exists an isometric embedding $l_{2;\pole}^m\ra l_{p;\pole}^n$.}

c) {\em There exists an $m$-dimensional Euclidean subspace in the
  normed space $l_{p;\pole}^n$.}

Thus, each row of our tables is an existence theorem which can
be formulated in any of equivalent form a), b), c) with some concrete
values $m,p,n$.

\begin{longtable}{|c|c|c|c|c|c|}
\caption{Input equalities  $n=N_{\pole}(m,p)$}\\\hline
&$\pole$&$m$&$p$&$n$  & \textbf{References}\\\hline
\endfirsthead\multicolumn{6}{c}
 { \textit{Continued from previous page}} \\\hline &$m$&$p$&$n$   &\textbf{ GUB}
& \textbf{References}
          \\\hline\endhead\hline
           \multicolumn{6}{r}{\textit{Continued on next page}} \\\endfoot\hline\endlastfoot
e1& $\mr$& 4& 4 & 11& \cite{R}\\\hline%e3
e2& $\mr$& 23 & 6 & 2~300  & \cite{DGS} \\\hline%e4
e3& $\mr$& 24& 10 & 98~280 & \cite{DGS}\\\hline%e5
e4& $\mc$ &2 &8 &10  & \cite{LSe} \\\hline%e9
e5& $\mc$& 2& 10  & 12  & \cite{LSe} \\\hline%e10
e6&$\mc$ & 4 & 6 & 40 &\cite{K} \\\hline%e7
e7& $\mc$& 6 & 6  & 126  & \cite{K}\\\hline%e8
e8&$\mh$&5&6&165&\cite{hoguni}  \\\hline%e12
\end{longtable}
\begin{longtable}{|c|c|c|c|c|c|}
\caption{Input inequalities $N_{\pole}(m,p)\leq n$}\\\hline
&$\pole$&$m$&$p$&$n$  & \textbf{References}\\\hline
\endfirsthead\multicolumn{6}{c}
 { \textit{Continued from previous page}} \\\hline &$m$&$p$&$n$   &\textbf{ GUB}
& \textbf{References}
          \\\hline\endhead\hline
           \multicolumn{6}{r}{\textit{Continued on next page}} \\\endfoot\hline\endlastfoot
i1& $\mr$ &4  &6 &23 &\cite{HadSlo94}\\\hline%i4
i2& $\mr$ &4&10&60&\cite{S}\\\hline%i15
%i3& $\mr$&4&14&264&\cite{HPV}\\\hline%i23
i3& $\mr$&4&18&360&\cite{S}\\\hline%i27
i4& $\mr$ &8&10&1200&\cite{HPV}\\\hline%i16
i5& $\mr$&8&12&12~120&\cite{HPV}\\\hline%i21
i6& $\mr$&8&14&13~200&\cite{HPV}\\\hline%i24
i7& $\mr$ &12 &6   &756& \cite{HPV}\\\hline%i5
i8& $\mr$ &12&8&4~032 &\cite{HPV}\\\hline%i11
i9& $\mr$ &12&10&25~200&\cite{HPV}\\\hline%i17
i10& $\mr$ &14 &4 &378  &\cite{BV}\\\hline%i1
i11& $\mr$ &14&6&756   &\cite{HPV}\\\hline%i6
i12& $\mr$ &14&8&44~982&\cite{HPV}\\\hline%i12
i13& $\mr$ &14&10&53~718&\cite{HPV}\\\hline%i18
i14& $\mr$ & 16 &6 &2~160  &\cite{hoguni}\\\hline%i7
i15& $\mr$ &16&8&32~780&\cite{HPV}\\\hline%i13
i16& $\mr$ &16&10&65~760&\cite{HPV}\\\hline%i19
i17& $\mr$&16&12&2~277~600&\cite{HPV}\\\hline%i22
i18& $\mr$ &20  &4& 1~980 &\cite{BV}\\\hline%i2
i19& $\mr$ &20&8&172~920&\cite{HPV}\\\hline%i14
i20& $\mr$ &20&10&2~263~800&\cite{HPV}\\\hline%i20
i21& $\mr$&24&14&8~484~840&\cite{HPV}\\\hline%i25
i22& $\mr$&24&16&207~501~840&\cite{HPV}\\\hline%i26
i23&$\mr$&24&18&2~522~192~400&\cite{HPV}\\\hline%i28
i24& $\mr$ &26&4  &  10~920 &\cite{BV}\\\hline%i3
i25& $\mr$ &26 &6  &21~840 &\cite{BV}\\\hline%i8
i26& $\mr$ &32 &6  &73~440 &\cite{BV}\\\hline%i9
i27& $\mr$ &36 &6 &164~160   &\cite{BV}\\\hline%i10
%i29& $\mc$&2&18&60&\cite{LSe}\\\hline%i32
i28&   $\mc$&9&4&90&\cite{hoguni}\\\hline%i29 i30
i29& $\mc$&12&10&32~760&\cite{hoguni}\\\hline%i31 i31
i30& $\mc$ &28&4&4~060&\cite{hoguni}\\\hline%i30 i32
i31&$\mh$&3&10&315&\cite{hoguni}\\\hline%i33 i33
\end{longtable}
\medskip

\begin{longtable}{|c|c|c|c|c|c|}
\caption{Results $N_{\mr}(m,p)\leq n$}\\\hline
~&$m$&$p$&$n$   &\textbf{ GUB}  & \textbf{References}\\\hline
\endfirsthead\multicolumn{6}{c}
 { \textit{Continued from previous page}} \\\hline &$m$&$p$&$n$   &\textbf{ GUB}
 & \textbf{References}
          \\\hline\endhead\hline
           \multicolumn{6}{r}{\textit{Continued on next page}} \\\endfoot\hline\endlastfoot

r0&4&14& 256&679&\eref{eq:nuovo}\\\hline
r1&4&16& 360&968&\eref{eq:redu}, i3\\\hline%r60 r2
r2&5&10&360&1000&\eref{eq:scociu}, i2\\\hline%r38 r4
r3&5&14&2~048&3~059&\eref{eq:uovo}\\\hline%r54 r5
r4& 5&16& 2881&4844&\eref{eq:scociu}, r1\\\hline%r61 r6
r5&5&18&3~600& 7~314&\eref{eq:scociu}, i3\\\hline%r65 r7
r6&6&8&615 &1286&\eref{eq:koly}, \eref{eq:scocit}, r1($\mc$) \\\hline%24 r8
r7&6&10&1296&3002& \eref{eq:koly}, \eref{eq:scocit}, r2($\mc$)\\\hline%r39 r9
r8&8&8&1~200& 6~434& \eref{eq:redu}, i4 \\\hline%r25 r11
r9&9&8&4~801&12~869&\eref{eq:scociu}, r8 \\\hline%r26 r12
r10&9&10& 7~200  &43~757&\eref{eq:scociu}, i4\\\hline%r40 r13
r11&9&12&72~721& 125~969&\eref{eq:scociu}, i5\\\hline%r52 r14
r12&9&14&105~600& 319~769& \eref{eq:scociu}, i6\\\hline%r55 r15
r13&10&6&1~280& 5~004 & \eref{eq:koly}, r4($\mc$)\\\hline%r5 r16
r14&10&8&19~205&24~309&\eref{eq:scociu}, r9 \\\hline%r27 r17
r15&10&10&43~200 &92~377&\eref{eq:scociu}, r10\\\hline%r41 r18
r16&11&6&5~120& 8~007 &\eref{eq:scociu}, r13\\\hline%r6 r19
r17&13&6&3~024& 18~563& \eref{eq:scociu}, i7\\\hline%r7 r20
r18&13& 8&16~129&125~969&\eref{eq:scociu}, i8\\\hline%r28 r21
r19&13&10&151~200& 646~645&\eref{eq:scociu}, i9\\\hline%r42 r22
r20&15&4&757   & 3~059& \eref{eq:scociu}, i10\\\hline%r1 r23
r21&15&6&3~024& 38~759&\eref{eq:scociu}, i11\\\hline%r8 r24
r22&15& 8&179~929&319~769&\eref{eq:scociu}, i12\\\hline%r29 r25
r23&15&10&322~308&1~961~255& \eref{eq:scociu}, i13\\\hline%r43 r26
r24&17&6&8~640& 74~612&\eref{eq:scociu}, i14\\\hline%r9 r27
r25&17& 8&131~121& 735~470&\eref{eq:scociu}, i15\\\hline%r30 r28
r26&17&10&394~560&5~311~734&\eref{eq:scociu}, i16\\\hline%r44 r29
r27&17&12&13~665~601&30~421~754& \eref{eq:scociu}, i17\\\hline%r53 r30
r28&18&6&34~560& 100~946&\eref{eq:scociu}, r24\\\hline%r10 r31
r29&18& 8&524~485&1~081~574&\eref{eq:scociu}, r25\\\hline%r31 r32
r30&18&10& 2~367~360&9~436~284&\eref{eq:scociu}, r26\\\hline%r45 r33
r31&20&6&3795&177~099&\eref{eq:koly}, i1, e8\\\hline%r11 r34
r32&21&4&3~961 &10~625&\eref{eq:scociu}, i18\\\hline%r2 r35
r33&21&6&15~180&230~229& \eref{eq:scociu}, r31\\\hline%r12 r36
r34&21& 8&691~681& 3~108~104 & \eref{eq:scociu}, i19\\\hline%r32 r37
r35&21&10&13~582~800&30~045~014&\eref{eq:scociu}, i20\\\hline%r46 r38
r36&22&4&7~923 &12~649&\eref{eq:scociu}, r32\\\hline%r3 r39
r37&22 &6&60~721   &296~009& \eref{eq:scociu}, r33\\\hline%r13 r40
r38&22& 8&2~766~725&4~292~144 &\eref{eq:scociu}, r34\\\hline%r33 r41
r39& 24& 4& 9~200&  17~549& \eref{eq:redu}, r40\\\hline%
r40& 24& 6& 9~200&  475~019& \eref{eq:scociu}, e2\\\hline%r14 r42
r41&24 & 8& 98~280&7~888~724 &\eref{eq:redu}, e3\\\hline%r34 r43
r42& 25& 6& 36~800& 593~774&\eref{eq:scociu}, r40\\\hline%r15 r44
r43& 25& 8& 393~121& 10~518~299&\eref{eq:scociu}, r41\\\hline%r35 r45
r44&25&10&  589~680&  131~128~139&\eref{eq:scociu}, e3 \\\hline%r47 r46
r45& 25& 12& 67~878~720&1~251~677~699& \eref{eq:redu}, r46\\\hline%
r46&25&14&67~878~720&9~669~554~099&\eref{eq:scociu}, i21\\\hline%r56 r47
r47&25&16&1~660~014~721&62~852~101~649& \eref{eq:scociu}, i22\\\hline%r62 r48
r48&25&18&25~221~924~000&353~697~121~049&\eref{eq:scociu}, i23\\\hline%r67 r49
r49& 26&8& 1~572~485& 13~884~155&\eref{eq:scociu}, r43\\\hline%r36 r50
r50&26&10& 3~538~080 & 183~579~395 & \eref{eq:scociu}, r44\\\hline%r48 r51
r51&26&12&543~029~760&1~852~482~995&\eref{eq:redu}, r52\\\hline%
r52&26&14&543~029~760&15~084~504~395&\eref{eq:scociu}, r46\\\hline%r57 r52
r53&26&16&13~280~117~769&103~077~446~705&\eref{eq:scociu}, r47\\\hline%r63 r53
r54&26&18&252~219~240~000&608~359~048~205&\eref{eq:scociu}, r48\\\hline%r68 r54
r55&27&4&21~841&27~404&\eref{eq:scociu}, i24\\\hline%r4 r55
r56& 27& 6&87~360&906~191&\eref{eq:scociu}, i25\\\hline%r16 r56
r57& 27&8&6~289~941&18~156~203& \eref{eq:scociu}, r49\\\hline%r37 r57
r58&27&10& 21~228~480& 254~186~855&\eref{eq:scociu}, r50\\\hline%r49 r58
r59&27&14&4~344~238~080&23~206~929~839&\eref{eq:scociu}, r52\\\hline%r58 r59
r60&27&16&106~240~942~153&166~509~721~601&\eref{eq:scociu}, r53\\\hline%r64 r60
r61&28& 6&349~440&1~107~567&\eref{eq:scociu}, r56\\\hline%r17 r61
r62&28&10& 38~918~880& 348~330~135&\eref{eq:koly}, i2, r7($\mh$)\\\hline%r50 r62
r63&28&14&34~753~904~640&35~240~152~719& \eref{eq:scociu}, r59\\\hline%r59 r63
r64&29&10&239~513~280&472~733~755&\eref{eq:scociu}, r62\\\hline%r51 r64
r65&33& 6& 293~760&2~760~680& \eref{eq:scociu}, i26\\\hline%r18 r65
r66&34& 6&1~175~040&3~262~622&\eref{eq:scociu}, r65\\\hline%r19 r66
r67&37& 6&656~640&5~245~785&\eref{eq:scociu}, i27\\\hline%r20 r67
r68&38&6&2~626~560&6~096~453&\eref{eq:scociu}, r67 \\%r21 r68
\end{longtable}
\medskip

\begin{longtable}{|c|c|c|c|c|c|}
\caption{Results $N_{\mc}(m,p)\leq n$}\\\hline
~&$m$&$p$&$n$   &\textbf{ GUB}  & \textbf{References}\\\hline
\endfirsthead\multicolumn{6}{c}
 { \textit{Continued from previous page}} \\\hline
&$m$&$p$&$n$   &\textbf{ GUB}  & \textbf{Source}
          \\\hline\endhead\hline
           \multicolumn{6}{r}{\textit{Continued on next page}} \\\endfoot\hline\endlastfoot
r0&2&18&50&99&\eref{eq:ubc2} \\\hline
r1&3&8&123&224&\eref{eq:scowc}, e4\\\hline%r15
r2&3&10& 216& 440&\eref{eq:scowc}, e5\\\hline%r19
r3& 3&18&2~500&3~024& \eref{eq:scowc}, r0\\\hline%r22
r4&5 &6&320 &1~224 & \eref{eq:scowc}, e6\\\hline%r6
r5& 7   &6&  1~008     & 7~055&\eref{eq:scowc}, e7\\\hline%r7
r6&8& 6&2~160&14~399&\eref{eq:koly}, i14\\\hline%r8
r7& 9& 6&  17~280& 27~224&\eref{eq:scowc}, r6\\\hline%r9
r8&10&4& 362   &3~024&\eref{eq:scowc}, i28\\\hline%r1
r9& 11&4 & 1450 & 4~355 &\eref{eq:scowc}, r8\\\hline%r2
r10 &12&4&5802&6~083  &\eref{eq:scowc}, r9\\\hline%r3
r11 &12& 6& 32~760&132~495& \eref{eq:redu}, r12\\\hline%r10
r12 &12&8&32~760& 1~863~224& \eref{eq:redu}, i29\\\hline%r16
r13& 13& 6&73~600& 207~024&\eref{eq:scsedi}, \eref{eq:simp}, r40($\mr$)\\\hline%r11
r14&13&8& 393~123&3~312~399& \eref{eq:scowc}, r12\\\hline%r17
r15&13&10&589~680& 38~291~343& \eref{eq:scowc}, i29\\\hline%r20
r16&14& 6&174~720&313~599&\eref{eq:scsedi}, \eref{eq:simp}, i25\\\hline%r12
r17&14&8&4~717~479&5~664~399&\eref{eq:scowc}, r14\\\hline%r18
r18&14&10&63~685~440& 73~410~623& \eref{eq:scsedi}, \eref{eq:scocit}, r50($\mr$)\\\hline%r21
r19&17&6&587~520&938~960&\eref{eq:scsedi}, \eref{eq:simp}, i26\\\hline%r13
r20&19&6&1~313~280&1~768~899&\eref{eq:scsedi}, \eref{eq:simp}, i27\\\hline%r14
r21&29 &4&  16~242  & 189~224&\eref{eq:scowc}, i30\\\hline%r4
r22 & 30&4&64~970&216~224  & \eref{eq:scowc}, r21\\%r5
\end{longtable}
\begin{longtable}{|c|c|c|c|c|c|}
\caption{Results $N_{\mh}(m,p)\leq n$}\\\hline
~&$m$&$p$&$n$   &\textbf{ GUB}  & \textbf{References}\\\hline
\endfirsthead\multicolumn{6}{c}
 { \textit{Continued from previous page}} \\\hline &$m$&$p$&$n$   &\textbf{ GUB}  & \textbf{Source}
          \\\hline\endhead\hline
           \multicolumn{6}{r}{\textit{Continued on next page}} \\\endfoot\hline\endlastfoot
r1&4&10&20~790&60~983&\eref{eq:scsedipart},  e1, i31\\\hline%r6
r2&5&4&165&  824&\eref{eq:redu}, e8\\  \hline%r1
r3& 6&4&1~324&1~715& \eref{eq:scsedipart}, \eref{eq:simp}, r2 \\  \hline%r2
r4&6&6&2~640&26~025&\eref{eq:scsedipart},  \eref{eq:simp}, e8\\  \hline%r4
r5&7&6&42~240&63~699&\eref{eq:scsedipart}, \eref{eq:simp}, r4\\\hline%r5
r6&7&10 & 6~486~480&8~836~463&  \eref{eq:scsedi},  i2, e3\\%r7
\end{longtable}

In conclusion let us note that

\begin{itemize}
\item r0($\mr$) improves $N_{\mr}(4,14)\leq 264$ from \cite{HPV},
\item r31($\mr$) improves $N_{\mr}(20,6)\leq 3~960$ from \cite{BV},
\item r39($\mr$) improves $N_{\mr}(24,4)\leq 13~104$ from \cite{BV},
\item r40($\mr$) improves $N_{\mr}(24,6)\leq 26~213$ from \cite{BV},
\item r0($\mc$) improves $N_C(2,18)\leq 60$ from \cite{LSe}.
\end{itemize}


\begin{thebibliography}{99}
\bibitem{S}
Salihov, G.N.,
Cubature formulas for a hypersphere that are invariant with respect to the
group of the regular 600-face, {\em Soviet Math. Doklady}, {\bf 16} (1975), 1046-1050.
\end{thebibliography}
\end{document}